\documentclass[12pt]{amsart}
\usepackage{graphics,wrapfig, graphicx, mathrsfs}
\newtheorem{lemma}{Lemma}[section]
\newtheorem{theorem}{Theorem}[section]

\newtheorem{proposition}{Proposition}[section]

\numberwithin{equation}{section} \numberwithin{theorem}{section}
\numberwithin{example}{section} \numberwithin{remark}{section}
\numberwithin{figure}{section} \numberwithin{algorithm}{section}

\def\ep{\varepsilon}
\def\ba{\begin{array}}
\def\ea{\end{array}}
\def\bma{\left(\begin{matrix}}
\def\ema{\end{matrix}\right)}
\def\be{\begin{equation}}
\def\ee{\end{equation}}

\def\vu{{\bf u}}

\def\vU{{\bf U}}
\def\vV{{\bf V}}
\def\vW{{\bf W}}
\def\vE{{\bf E}}
\def\vUd{\tilde{\vU}}

\def\nc{\nabla\!\!\cdot\!}
\def\cn{\!\cdot\!\!\nabla}
\def\pa{\partial}
\def\tpa{{\partial}}
\def\pt{\partial_t}
\def\po{\partial_0}
\def\ptt{\partial_{tt}}
\def\px{\partial_1}
\def\py{\partial_2}

\def\dfrac{\displaystyle\frac}
\def\half{\dfrac{1}{2}}
\def\cL{{\mathcal L}}
\def\cA{{\mathcal A}}
\def\mR{\mbox{R}}

\def\veps{\varepsilon}
\def\Ga{\Gamma^{\alpha}}
\def\Gb{\Gamma^{\beta}}
\def\Gg{\Gamma^{\gamma}}
\def\bega{{\beta\gamma}}
\def\LGN{|}
\def\RGN#1{|_{#1}}
\def\LGNN{\|}
\def\RGNN#1{\|_{H^{#1}_{\Gamma}}}
\def\gaE{F}

\def\Hw{H}

\def\lang{\left\langle}
\def\rang{\right\rangle}

\def\rhoK{\rho^K}

\def\vuK{{\bf u}^K}

\def\vS{{\bf S}}

\begin{document}

\title[Rotating Shallow Water System]{On the Classical  Solutions
of Two Dimensional Inviscid  Rotating Shallow Water System}

\author[Bin Cheng]{Bin Cheng}
\author[Chunjing Xie]{Chunjing  Xie}
\address{\newline
        Department of Mathematics,
	University of Michigan\newline
	530 Church St.
	Ann Arbor, MI 48109 USA}
\email[Bin Cheng]{bincheng@umich.edu}
\email[Chunjing Xie]{cjxie@umich.edu}

\date\today
\keywords{Rotating Shallow Water system; Klein-Gordon equations; classical solutions; global existence; symmetric system of hyperbolic PDEs.}
\subjclass{35L45 (Primary) 76N10, 35L60 (Secondary)}
\begin{abstract} 
We prove global existence  and asymptotic
behavior of classical solutions for two dimensional inviscid
Rotating Shallow Water system with small initial data subject to the
zero-relative-vorticity constraint. One of the key steps is a
reformulation of the problem into a symmetric quasilinear
Klein-Gordon system, for which the global existence of classical solutions is
then proved with combination of the vector field approach and the
normal forms, adapting ideas developed in \cite{OzawaSL}. We also probe the case of general initial
data and reveal a lower bound for the lifespan that is almost inversely
proportional to the size of the initial relative vorticity.\end{abstract}

\maketitle

\section{Introduction and Main Results}

The system of Rotating Shallow Water (RSW) equations is a widely
adopted 2D approximation of the 3D incompressible Euler equations
and the Boussinesque equations in the regime of large scale
geophysical fluid motion (\cite{Pedlosky}). It is also regarded as
an important extension of the compressible Euler equations with
additional rotational forcing.

Start with the following formulation,
\begin{align}
\label{RSWh}\pt h+\nc(h\vu)&=0,\\
\label{RSWu} \pt\vu+\vu\cn\vu+\nabla h+\vu^\perp&=0,
\end{align}
where $h=h(t,x_1,x_2)$  and $ \vu=(u_1(t,x_1,x_2),u_2(t,x_1,x_2))^T$ denote the total height and
velocity of the fluids, respectively, and $\vu^\perp:=(-u_2, u_1)^T$
corresponds to the rotational force. For mathematical convenience,
all physical parameters are scaled to the unit (cf. \cite{Majda} for
detailed discussion on scaling).

Since $(h, \vu)=(1,0)$ is a steady-state solution of  (\ref{RSWh}), (\ref{RSWu}), we introduce the perturbations
$(\rho,\vu):=(h-1,\vu)$ and arrive at
\begin{align}
\label{RSWrho}\pt\rho+\nc(\rho\vu)+\nc\vu&=0,\\
\label{RSWurho} \pt\vu+\vu\cn\vu+\nabla\rho+ \vu^\perp&=0,
\end{align}
subject to initial data
\be\label{RSWinit}\quad\rho(0,\cdot)=\rho_0,\;\vu(0,\cdot)=\vu_0.\ee

An important feature of the RSW system is that the relative
vorticity \(\theta:=\nabla\times\vu-\rho=(\px u_2-\py u_1)-\rho\) is
convected by $\vu$, \be\label{theta}\pt\theta+\nc(\theta\vu)=0.\ee
Indeed, $\nabla\times$(\ref{RSWurho})$-$(\ref{RSWrho}) readily leads
to (\ref{theta}). The linearity of (\ref{theta}) then suggests that
$\theta\equiv0$ be an invariant with respect to time (as long as
$\vu\in C^1$), i.e. \be\label{inv}
\theta_0\equiv0\iff\theta(t,\cdot)\equiv0\iff\nabla\times\vu\equiv\rho.
\ee

Before stating the main theorems, we fix some notations. For $1\leq
p\leq \infty$, let $L^p$ denote the standard $L^p$ space on $\mR^2$.
For $l \geq 0$ and $s\geq 0$, define the weighted Sobolev norm
associated with the space $\Hw^{l,s}$ as
\begin{align}\label{def:weighted}
\|v\|_{\Hw^{l,s}}:=\|(1+|x|^2)^{s/2}(1-\Delta)^{l/2} v\|_{L^2}.
\end{align}
Also, denote the standard Sobolev space $H^l:=\Hw^{l,0}$.

\begin{theorem}\label{mainTh}
Consider the RSW system (\ref{RSWrho}), (\ref{RSWurho}), (\ref{RSWinit}) with initial data $\vu_0=(u_{1,0}, u_{2,0})^T\in \Hw^{k+2,k}$ for $k\ge52$ and zero relative vorticity,
\[
\rho_0=\pa_1 u_{2,0}-\pa_2 u_{1,0}.
\]
Then, there exists a universal constant $\delta_0>0$ such that the
RSW system admits a unique
classical solution $(\rho, \vu)$ for all time, provided that the
initial data satisfy
\begin{equation*}\label{Isizeu}
\|\vu_0\|_{\Hw^{k+2,k}}=\delta<\delta_0.
\end{equation*}
Moreover, there exists a free solution $\vu^+(t,\cdot)$ such that  %\be\label{asy} 
\[\|\vu(t,\cdot)-\vu^+(t,\cdot)\|_{H^{k-15}}+\|\pa_t\vu(t,\cdot)-\pa_t\vu^+(t,\cdot)\|_{H^{k-16}}\leq C(1+t)^{-1},\]
where
$\vu^+(t,\cdot):=(\cos(1-\Delta)^{1/2}
t)\vu_0^++\left((1-\Delta)^{-1/2}\sin((1-\Delta)^{1/2}
t)\right)\vu_1^+$ for some $\vu_{0}^+\in H^{k-15}$ and $\vu_1^+\in H^{k-16}$.
\end{theorem}
The proof is a straightforward combination of Lemma \ref{symm}, Theorem \ref{KGsystem} and Theorem \ref{asymTh} below.

This result shows fundamentally different lifespan of the classical
solutions for the RSW system in comparison with the compressible
Euler systems. Note that the relative vorticity $\theta=\nabla\times\vu-\rho$ in the RSW
equations plays a very similar role as vorticity $\nabla\times\vu$ in the compressible
Euler equations. Correspondingly, RSW solutions with zero relative
vorticity $\theta$ is an analogue of irrotational solutions for the
compressible Euler equations. However, the life span for 2D
compressible Euler equations with zero vorticity was proved to be
bounded from below (Sideris \cite{Sideris:2D}) and above (Rammaha
\cite{Rammaha}) by $O({1}/{\delta^2})$. Here, $\delta$ indicates the
size of the initial data. Sideris also showed that the life span in
the 3D case is bounded from below by $O(e^{1/\delta})$ in
\cite{Sideris:3D} and from above by $O(e^{1/{\delta^2}})$ in
\cite{Sideris:3D:singularity}. Our result for 2D RSW system, on the other hand, is global in time due to the additional rotating
force. Consult \cite{Babin, LiTa:rotation, ChTa:SIAM} for related results on global and long-time existence of classical RSW solutions in various regimes.

A key ingredient of the proof is to treat the RSW system as a system
of quasilinear Klein-Gordon equations -- cf. Lemma \ref{symm} for a formal
discussion. Such reformulation allows us to utilize the fruitful
results on nonlinear Klein-Gordon equations appearing in recent decades. To
mention a few, for spatial dimensions $N\ge5$, Klainerman and
Ponce \cite{KlainermanP} and Shatah \cite{Shatahev} showed that the
Klein-Gordon equation admits a unique, global solution for small
initial data and that the solution approaches the free solution of
the linear Klein-Gordon equation as $t\rightarrow \infty$. The
proofs in \cite{KlainermanP, Shatahev} are based on $L^p-L^q$ decay
of the linear Klein-Gordon equations. The global existence for  quasilinear Klein-Gordon equations in dimensions $N=4,3$ was proved
independently by Klainerman in \cite{Klainerman} using the vector
fields approach and Shatah in \cite{Shatah} using the normal forms.
In the $N=2$ case, global existence of classical solutions become increasingly subtle due
to the $(1+t)^{-1}$ decay rate of solutions to linear Klein-Gordon equations. Nevertheless, it has been proved by Ozawa et al in
\cite{OzawaSL} for semilinear, scalar equations. The authors combined
the vector fields approach and the normal form method after partial
results in \cite{GeorgievP, Georgiev, Simon}. The result of global existence on
quasilinear, scalar Klein-Gordon equations was announced in
\cite{OzawaQL}. Recently, Delort et al obtained global existence for
a two dimensional system of two Klein-Gordon equations in
\cite{Fang}, where the authors transformed the problem using
hyperbolic coordinates and then studied it with the vector fields
approach, which was restricted to compactly supported initial data.
For applications of the Klein-Gordon equations in fluid equations,
we refer to \cite{Guo} by Y. Guo on global existence of three
dimensional Euler-Poisson system. Note that the irrotationality
condition used there plays a counterpart of the
zero-relative-vorticity constraint in our result.

For general initial data, we have the following theorem on the
lifespan of classical solutions. Its proof is given in Section 5.
\begin{theorem}\label{perTh}
Consider the RSW system (\ref{RSWrho}), (\ref{RSWurho}), (\ref{RSWinit}) with initial data $(\rho_0,\vu_0)\in \Hw^{k+1,k}$ for $k\ge52$. Let $\delta$ denote the size of the initial data \begin{equation*}
\delta=\|(\rho_0,\vu_0)\|_{\Hw^{k+1,k}},
\end{equation*}
 and $\veps$ the size of the initial relative vorticity,
 \[\veps=\|(\pa_1 u_{2,0}-\pa_2 u_{1,0})-\rho_0\|_{H^2}.\]
Then, there exists a universal constant $\delta_0>0$ such that, for any $\delta\le\delta_0$, the
RSW system admits a unique
classical solution $(\rho, \vu)$ for \be\label{span:general}t\in[0,
C_1\veps^{-\frac{1}{1+C_2\delta}}].\ee
Here, $C_1$ and $C_2$ are constants independent of $\delta$ and $\veps$. 
\end{theorem}

This theorem confirms the key role that relative vorticity plays in the studies of Geophysical Fluid Dynamics (\cite{Pedlosky}). In fact, having two uncorrelated scales $\veps$ and $ \delta$ in (\ref{span:general})
allows us to solely let the size of the initial relative vorticity
$\ep\to0$ and achieve a very long lifespan of classical solutions, regardless of the total size of initial data. The proof, given in Section \ref{sec:general}, treats the full RSW system as perturbation to the zero-relative-vorticity one and utilizes the standard energy methods for symmetric hyperbolic PDE systems. The sharp estimates of Theorem \ref{mainTh} play a crucial role in controlling the total energy growth. We note that a similar problem of the compressible Euler equations is studied by Sideris in \cite{Sideris:2D}.

We also note by passing that the study of hyperbolic PDE systems
with small initial data is closely related, if not entirely
equivalent, to the singular limit problems with large initial data.
See \cite{Majda}, \cite{Babin} and references therein for results on
the particular case of inviscid RSW equations. A recent survey paper
\cite{Bresch} contains a collection of open problems and recent
progress on viscous Shallow Water Equations and related models.

We finally comment that all results in this paper should be true for two dimensional compressible Euler equations with rotating force and general pressure law. The proof remains largely the same except for the symmetrization part and associated energy estimates.

The structure of the rest of the paper is outlined as following. In
Section 2, we reformulate the RSW system into a symmetric hyperbolic
system of first order PDEs. Under the zero-relative-vorticity
constraint, it is further transformed into a system of quasilinear
Klein-Gordon equations with symmetric quasilinear part. Section 3 is
devoted to the local wellposedness of the RSW equation with general
initial data and zero-relative-vorticity initial data. Section 4
contains the proof of Theorem \ref{mainTh} in a series of lemmas,
adapting results from \cite{Georgiev}, \cite{Shatah}, \cite{OzawaSL}.
The discussion and proof of Theorem \ref{perTh} on general initial
data can be found in Section 5. The Appendix contains the proof of a
technical proposition used in Section 4.

\section{Reformulation of the Problem}

In order to obtain local wellposedness for (\ref{RSWrho}) and
(\ref{RSWurho}), we first symmetrize the system into a symmetric
hyperbolic system. This will also be used for proving global
existence, where we need to reduce (\ref{RSWrho}) and(\ref{RSWurho})
to a symmetric quasilinear Klein-Gordon system.

Introduce a symmetrizer $m:=2\left(\sqrt{1+\rho}-1\right)$ such that
$\rho=m+{1\over4}m^2$, then (\ref{RSWrho}), (\ref{RSWurho}) are
transformed into a symmetric hyperbolic PDE system,
\begin{align}
\label{RSWm}\pt m+\vu\cn m+\half m\nc\vu+\nc\vu&=0,\\
\label{RSWum} \pt\vu+\vu\cn\vu+\half m\nabla m+\nabla m+\vu^\perp&=0.
\end{align}

The following lemma asserts that, under the invariant (\ref{inv}),
the above system amounts to a system of Klein-Gordon equations with
symmetric quasilinear terms.

\begin{lemma}\label{symm}
Under the invariant (\ref{inv}), $\nabla\times\vu=\rho$, and
transformation $m=2\left(\sqrt{\rho+1}-1\right)$, the solution to
the RSW system (\ref{RSWrho}), (\ref{RSWurho}) satisfies the
following symmetric system of quasilinear Klein-Gordon equations for
$\vU:=\bma m\\\vu\ema$, \be\label{ptt}
\ptt\vU-\Delta\vU+\vU=\sum_{i,j=1}^2A_{ij}(\vU)\pa_{ij}\vU +
\sum_{j=1}^2A_{0j}(\vU)\pa_{0j}\vU+R(\vUd\otimes\vUd), \ee where
linear functions $A_{ij}$ and $A_{0j}$ map $\mR^3$ vectors to
\emph{symmetric} $3\times 3$ matrices and satisfy $A_{ij}=A_{ji}$.
The remainder term $R$ depends linearly on the tensor product
$\vUd\otimes\vUd$ with $\vUd:=(\vU^T,\pt\vU^T,\px\vU^T,\py\vU^T)$.
\end{lemma}

Here and below, for notational convenience, we use both $\pt$ and
$\pa_0$ to denote the time derivatives.
\begin{proof}
Rewrite (\ref{RSWm}), (\ref{RSWum}) into a matrix-vector form,
\be
\label{RSWmatrix}\pt\vU+\sum_{a=1,2}(u_aI+\half mJ_a)\pa_a\vU=\cL(\vU),
\ee
where
\be\label{defJL}
J_1:=\bma0&1&0\\1&0&0\\0&0&0\ema,\,\,\,\, J_2:=\bma0&0&1\\0&0&0\\1&0&0\ema,
\,\,\, \, \text{and}\,\, \cL(\vU):=-\bma \nc\vu\\
\nabla m+\vu^\perp\ema. \ee

By taking time derivative on the above system, we have
\[\ptt\vU+N(\vU)=\cL^2(\vU),\]
where the nonlinear term
\be\label{N1N2}\begin{split}
N(\vU)&=\pt\sum_{a=1,2}(u_aI+\half
mJ_a)\pa_a\vU+\cL\left(\sum_{a=1,2} (u_aI+\half
mJ_a)\pa_a\vU\right)\\&:=N_1+N_2.
\end{split}\ee
The $\cL^2$ term, using the  calculus identity
$\nabla(\nabla\cdot\vu) -\nabla^\perp(\nabla\times\vu)=\Delta\vu$,
is
\be\label{N3}\begin{split}
\cL^2(\vU)&=\bma \nc(\nabla m+\vu^\perp)\\ \nabla(\nc\vu)+
(\nabla m+\vu^\perp)^\perp\ema\\
&=\bma(\Delta-1)m -(\nabla\times\vu-m)\\
(\Delta-1)\vu+\nabla^\perp(\nabla\times\vu-m)\ema\\
&=(\Delta-1)\vU+\bma-(\nabla\times\vu-\rho+{1\over4} m^2)\\
\nabla^\perp(\nabla\times\vu-\rho+{1\over 4} m^2)\ema \qquad
\mbox{since  } \rho=m+{1\over4}m^2 \\
&=(\Delta-1)\vU+{1\over4}\bma-m^2\\
\nabla^\perp(m^2)\ema\qquad\mbox{ under (\ref{inv})}\\
&:=(\Delta-1)\vU+N_3.\end{split}\ee

Now that we've revealed the Klein-Gordon structure of (\ref{ptt}),
it suffices to show that the other terms $N_1,N_2,N_3$ can all be
split into a symmetric second order part and a remainder lower order
part as given in (\ref{ptt}).
\begin{itemize}
\item The $N_1$ term in (\ref{N1N2}). It is easy to see that the lower
order terms (with less than second order derivatives) in $N_1$ are
quadratic in $\vUd$, i.e. linear in $\vUd\otimes\vUd$. The terms
with second order derivatives are
\[\sum_{a=1,2}(u_aI+\half mJ_a)\pt\pa_a\vU\]
where matrices $I$, $J_a$ are all symmetric. 
\item The $N_2$ term in (\ref{N1N2}).
Observe that the linear operator $\cL$ partially comes from a
linearization of the nonlinear terms in (\ref{RSWmatrix}) and it
indeed can be represented as
\[\cL(\vU)=\sum_{a=1,2}J_a\pa_a\vU+K\vU\]
with constant matrix $K$. Thus, manipulate the $N_2$ term,
\[\begin{split}
N_2&=\sum_{a=1,2}J_a\pa_a\left(\sum_{b=1,2}(u_bI+\half mJ_b)\pa_b\vU\right)
+K\left(\sum_{b=1,2}(u_bI+\half mJ_b)\pa_b\vU\right)\\
&=\sum_{a=1,2}\sum_{b=1,2}(u_bJ_{a}+\half mJ_{a}J_b)\pa_{a}\pa_b\vU+
\mbox{ quadratic terms of }\vUd.
\end{split}\]
The quasilinear terms above have the desired symmetric structure
since for each index pair $(a,b)$, the coefficient of
$\pa_{a}\pa_b\vU=\pa_b\pa_{a}\vU$ is
\[
\frac{1}{2}(u_bJ_{a}+\half
mJ_{a}J_b)+\frac{1}{2}(u_{a}J_{b}+\half mJ_{b}J_{a})
\]
which, by the definition of $J_a$, is symmetric.
\item The $N_3$ term in (\ref{N3}). By
definition, this term has no second order derivatives and is
quadratic in $\vUd$.
\end{itemize}
\end{proof}

\section{Local Wellposedness}

As in the previous section, let
$\pa_0=\frac{\pa}{\pa t}$, $\pa_1=\frac{\pa}{\pa x_1}$, $\pa_2=
\frac{\pa}{\pa x_2}$. Define the vector fields
\begin{align}\label{def:vfields}
\Gamma:=\{\Gamma_j\}_{j=1}^6=\{\pa_0,\pa_1,\pa_2, L_1, L_2,
\Omega_{12}\},
\end{align}
where
\begin{align*}
L_j:=x_j\pa_t+t\pa_j,\,\,j=1,2; \quad
\Omega_{12}:=x_1\pa_2-x_2\pa_1.
\end{align*}
We abbreviate
\begin{align*}
\pa^{\alpha}=\pa_t^{\alpha_1}\pa_1^{\alpha_2}\pa_2^{\alpha_3}
\quad\mbox{ for } \alpha=(\alpha_1,\alpha_2,\alpha_3),
\end{align*}
and
\begin{align*}
\Gamma^{\beta}=\Gamma_1^{\beta_1}\cdots\Gamma_6^{\beta_6}
\quad\mbox{ for } \beta=(\beta_1,\cdots, \beta_6).
\end{align*}

The local wellposedness of (\ref{RSWm}) and (\ref{RSWum})
and their regularity are contained in the following theorem.
\begin{theorem}\label{LocalTh}
\begin{itemize}
\item[(i)] Let $(m_0,\vu_0)\in H^n$ with $n\ge3$. Then, there exists a $T>0$
depending only on $\|(m_0,\vu_0)\|_{H^3}$ such that (\ref{RSWm}), (\ref{RSWum})
admit a unique solution $\vU=\bma m \\ \vu\ema $  on $[0, T]$ satisfying
\begin{align}
\vU\in  \bigcap_{j=0}^{n} C^j([0,T], H^{n-j}).
\end{align}
\item[(ii)] Under the assumptions in (i), also assume
$\rho_0=\pa_1 u_{2,0}-\pa_2 u_{1,0}$. Then,
\begin{align}\label{Zerovor2}
\rho(t,\cdot)=\pa_1 u_2(t,\cdot)-\pa_2 u_1(t,\cdot)\,\, \text{for}\,\, t\in[0,T],
\end{align}
and $\vU=\bma m\\\vu\ema=\bma2\left(\sqrt{\rho+1}-1\right)\\\vu\ema$
satisfies the Klein-Gordon system (\ref{ptt}). If the initial data
belong to the weighted Sobolev space (cf. definition
(\ref{def:weighted})) such that $\vU_0\in \Hw^{k+1,k}$ with $k\geq
3$, then the above solution $\vU$ satisfies
\begin{align}\label{locvec}
\Ga\vU,\,\Ga\pa\vU\in C([0,T], L^2),\qquad(1+|x|)\Gb\vU\in C([0,T],L^\infty),
\end{align}
for any multi-indices $|\alpha|\leq k$ and $|\beta|\leq k-3$.
\end{itemize}
\end{theorem}

\begin{proof}
The proof of (i) follows from the standard local wellposedness and
regularity theory for symmetric hyperbolic system, cf. \cite{Kato,
Majda84}.

For part (ii), (\ref{Zerovor2}) comes from the derivation of
(\ref{inv}). Then, it follows from Lemma \ref{symm} that $\vU$
solves (\ref{ptt}). Finally, the proof of (\ref{locvec}) is based on
the arguments in \cite{KlainermanP, Shatahev}. Note that, using
(\ref{RSWm}),  (\ref{RSWum}), one has
\begin{align*}
  \vU(0,\cdot)\in \Hw^{k+1,k} \text{ and } \pa_t^l \vU(0,\cdot)\in \Hw^{k+1-l,k}
   \end{align*}
as long as $\vu_0\in H^{k+2,k}$ and $\rho_0=\nabla\times\vu_0$.
\end{proof}

\section{Global Existence and Asymptotic Behavior with Zero Relative Vorticity}

Throughout this section, we focus on the solutions with zero
relative vorticity -- cf. (\ref{inv}). Theorem \ref{LocalTh}, part
(ii), suggests that the RSW system be treated as a system of
quasilinear Klein-Gorden equations. To this end, it is convenient to
introduce the following generalized Sobolev norms associated with
vector fields $\Gamma$ defined in (\ref{def:vfields}).
\begin{align*}
\LGN\vU\RGN{l,d}(t):=&
\sum_{|\alpha|\leq l }|(1+t+|x|)^{-d}\Gamma^{\alpha} \vU(t,x)|_{L^{\infty}_x},\\
\LGNN\vU\RGNN{l}(t):=&\sum_{|\alpha|\leq l}
\|\Gamma^{\alpha} \vU(t,x) \|_{L^2_x}.
\end{align*}

To extend the local solution of (\ref{RSWm}) and (\ref{RSWum})
globally in time, we need to derive the decay and energy estimates
of solutions to (\ref{ptt}). We start with
defining a functional (see e.g. \cite{OzawaSL}) measuring the size of the solution at time
$t\ge0$,\be\label{def:X}\begin{split}
X(t):=&\sup_{s\in[0,t]}\left\{\LGN\vU\RGN{k-25,-1}(s)+\LGNN\vU\RGNN{k-9}(s)+\LGNN\tpa \vU\RGNN{k-9}(s)\right.\\
&\left.+(1+s)^{-\sigma}\LGNN\vU\RGNN{k}(s)+(1+s)^{-\sigma}\LGNN\tpa
\vU\RGNN{k}(s)\right\},
\end{split}\ee
here, pick any fixed $\sigma\in(0,1/2)$ and $k\ge52$.

We then state and prove the following global existence result regarding any symmetric quasilinear system of Klein-Gordon equations in 2D. Two key lemmas used in the proof will be discussed immediately after this.
\begin{theorem}\label{KGsystem}Consider a two dimensional $n$ by $n$ system (\ref{ptt}) satisfying the conclusion of Lemma \ref{symm}. Then, for any $k\ge52$, there exists a universal constant $\delta_0$ such that the system admits a unique classical solution for all times if
\begin{equation*}
\|\vU_0\|_{\Hw^{k+1,k}}=\delta<\delta_0.
\end{equation*}
In particular, $X(t)\le C\delta$ uniformly for all positive times.
\end{theorem}
\begin{proof}By the definition of $X(t)$ and local existence (\ref{locvec}) of Theorem \ref{LocalTh}, there exists $T$ such that
\be\label{XT}
  X(T)\leq 4C_1\delta.
 \ee
 Here, we choose constant $C_1$ to be greater than all constants appearing in Lemma \ref{lemL} and \ref{lemenergy} below. Then, choose $\delta$ to be sufficiently small so that the assumptions of Lemma \ref{lemL} and \ref{lemenergy} are satisfied, which in turn implies
\[
X(T)\leq 2C_1\delta+32C_1^3\delta^2.
\]
Impose one more smallness condition on $\delta$ so that $X(T)\le 3C_1\delta$ in the above estimate. Finally, by the continuity argument, we can extend $T$ in (\ref{XT}) to infinity, i.e., have $X(t)\le 4C_1\delta$ uniformly for all positive times. 
\end{proof}

The following lemmas provide estimates on the lower order norms and highest order norms of $X(t)$ respectively. The quadratic term $X^2(t)$, rather than linear term, on the RHS of these estimates guarantees that we can extend such estimates to global times as long as $X(t)$ stays sufficiently small.
\begin{lemma}\label{lemL}
Assume
$\|\vU_0\|_{\Hw^{k+1,k}}\le1$, $X(t)\le1$. Then, the solution $\vU$
of (\ref{ptt}) satisfies \be\label{Lestimate}
\LGN\vU\RGN{k-25,-1}(t)+\LGNN\vU\RGNN{k-9}(t)+\LGNN\pa\vU\RGNN{k-9}(t)\leq
C(\|\vU_0\|_{\Hw^{k+1,k}}+ X^2(t)). \ee as long as the solution
exists. Here, constant $C$ is independent of $\delta$ and $t$.
\end{lemma}
\begin{proof}
We start the proof with defining the bilinear form associated with
kernel $Q(y,z)$,
\[%\label{def:form}
\begin{split}
[G,Q,H](x):=&\int_{\mR^2\times\mR^2}G^T(y)Q(x-y,x-z)H(z)\,dydz\\
=&{1\over(2\pi)^4}\int_{\mR^2\times\mR^2}e^{i(\xi+\eta)\cdot
x}\hat{G}^T(\xi)\hat{Q}(\xi,\eta)\hat{H}(\eta)\,d\xi
d\eta.
\end{split}
\]
Here, $G(\cdot)$, $H(\cdot)$ are any $(2\times1)$-vector-valued
functions defined on $\mR^2$ and $Q(\cdot,\cdot)$ is $(2\times
2)$-matrix-valued distribution defined on $\mR^2\times\mR^2$.
Fourier transform is denoted with $\hat{\;\;}$ for both $\mR^2$ and
$\mR^2\times\mR^2$.

The same notation will be used for scalars
\[
[g,q,h](x):=\int_{\mR^2\times\mR^2}g(y)q(x-y,x-z)h(z)\,dydz.
\]

Then, we follow the construction of \cite{Shatah} to transform
(\ref{ptt}) in terms of the new variable
\be\label{def:V}\vV=(V_1,V_2,V_3)^T=\vU+\vW=\vU+(W_1,W_2,W_3)^T\ee
where
\be\label{def:vi}
W_k:=\sum_{i,j=1}^3\left[\bma U_i\\
\pt U_i\ema,Q^{ij}_k,\bma U_j\\
\pt U_j\ema\right],\quad k=1,2,3.
\ee
The kernels $Q^{ij}_k$ are to be determined later so that the new
variable $\vV$ satisfies a Klein-Gordon system with \emph{cubic}
nonlinearity for which estimate (\ref{Lestimate}) will be proved
using techniques from \cite{OzawaSL}, \cite{Georgiev}.

Without loss of generality, we will demonstrate the proof using
$V_1,U_1,W_1,Q^{ij}_1$ associated with the mass equation. From now
on, the subscript ``1'' is neglected for simplicity.

\emph{Step 1.} We claim that there exists kernels $Q^{ij}$ in
(\ref{def:vi}) such that $V=U+W$ satisfies the following
Klein-Gordon equation with cubic and quadruple nonlinearity,
\be\label{cubicKG}(\ptt-\Delta +1)V=S\ee where the
RHS
\be\label{cubicRHS}
\begin{split}S:=&\sum_{|\alpha|+|\beta|+|\gamma|\le4\atop
\max\{|\alpha|,|\beta|,|\gamma|\}\le3}\sum_{a,b,c=1}^3[\pa^\alpha
U_a\pa^\beta U_b,q_{\alpha\beta\gamma}^{abc},\pa^\gamma
U_c]\\&+\sum_{|\alpha|+|\beta|+|\gamma|+|\zeta|\le4\atop
\max\{|\alpha|,|\beta|,|\gamma|,|\zeta|\}\le3}\sum_{a,b,c,d=1}^3[\pa^\alpha
U_a\pa^\beta U_b,q_{\alpha\beta\gamma\zeta}^{abcd},\pa^\gamma
U_c\pa^\zeta U_d] \end{split}\ee with $q_{\alpha\beta\gamma}^{abc}$,
$q_{\alpha\beta\gamma\zeta}^{abcd}$ being linear combinations of the
entries of all $Q_k^{ij}$'s. Moreover, all the $Q_k^{ij}$'s  satisfy
the growth
condition\be\label{growth:Q}\left|D^N\hat{Q}(\xi,\eta)\right|\le
C_N(1+|\xi|^4+|\eta|^4)\ee for any nonnegative integer N.

Indeed, substitute $V$ on the LHS of (\ref{cubicKG}) with $V=U+W$
where $W$ is defined in (\ref{def:vi}) for $k=1$ and $U$ satisfies
the first equation of (\ref{ptt}),
\be\label{form:0}(\ptt-\Delta+1)V=(\ptt-\Delta+1)U+(\ptt-\Delta+1)W.\ee

Then, apply the normal form transform on each term of the RHS of the
above equation. By (\ref{ptt}), the
$U$ term amounts to
\be\label{form:1}(\ptt-\Delta+1)U=\sum_{i,j=1}^3\left[\bma U_i\\\pt
U_i\ema,P^{ij},\bma U_j\\\pt U_j\ema\right]\ee with
$\hat{P}^{ij}(\xi,\eta)$ being $2\times2$ matrices of polynomials
with degree less than or equal to 2. By (\ref{def:vi}), the $W$ term amounts to
\be\label{form:2}(\ptt-\Delta+1)W=\sum_{i,j=1}^3\left[\bma U_i\\\pt
U_i\ema,\cA Q^{ij},\bma U_j\\\pt U_j\ema\right]+S\ee with $S$
satisfying (\ref{cubicRHS}) and linear transform $\cA$ defined as
\be\label{form:3}\widehat{\cA
Q}(\xi,\eta):=\bma0&|\xi|^2+1\\-1&0\ema
\hat{Q}\bma0&-1\\|\eta|^2+1&0\ema+(2\xi\cdot\eta-1)\hat{Q}\ee

Combining (\ref{form:0}) --- (\ref{form:3}), we find that, for
proving (\ref{cubicKG}) --- (\ref{growth:Q}), it suffices to show
that there exist solutions $\hat{Q}^{ij}(\xi,\eta)$ to
\be\label{eq:Q}\hat{P}^{ij}(\xi,\eta)+\widehat{\cA
{Q}}^{ij}(\xi,\eta)\equiv0\ee that satisfy the growth condition
(\ref{growth:Q}). This part of the calculation only involves basic
Linear Algebra and Calculus so we neglect the details.

\emph{Step 2.} We apply the decay estimate of Georgiev in
\cite{Georgiev} to obtain the $L^\infty$ estimate for $V$ and
therefore $U$ with $(1+t)^{-1}$ decay in time.
\begin{theorem}[{\cite[Theorem 1]{Georgiev}}]\label{thm:Georgiev}Suppose $u(t,x)$
is a solution of\[(\ptt -\Delta +1)u=F(t,x).\]Then, for $t\ge0$, we
have
\[
\begin{split}|(1+t+|x|)u(t,x)|\le& C\sum_{n=0}^\infty\sum_{|\alpha|\le4}
\sup_{s\in(0,t)}\phi_n(s)\|(1+s+|y|)\Ga f(s,y)\|_{L^2_y}\\
&+C\sum_{n=0}^\infty\sum_{|\alpha|\le5}\|(1+|y|)\phi_n(y)\Ga
u(0,y)\|_{L^2_y}.
\end{split}
\]
Here, $\{\phi_n\}_{n=0}^\infty$ is a Littlewood-Paley partition of
unity,
\[\sum_{n=0}^\infty\phi_n(s)=1,\quad s\ge0;\quad\phi_n\in C^\infty_0(\mR),
\quad\phi_n\ge0\quad\mbox{for all }n\ge0\]
\[\mbox{supp}\,\phi_n=[2^{n-1},2^{n+1}]\quad\mbox{for }n\ge1,
\quad\mbox{supp}\,\phi_0\cap\mR_+=(0,2].\]
\end{theorem}

Apply $\Ga$ on the Klein-Gordon equation (\ref{cubicKG}) and use the
commutation properties of the vector fields to obtain \((\ptt
-\Delta +1)\Ga V=\Ga S\) so that by the above theorem,
\[
\begin{split}|(1+t+|x|)\Ga V(t,x)|\le& C\sum_{n=0}^\infty
\sum_{|\beta|\le|\alpha|+4}
\sup_{s\in(0,t)}\phi_n(s)\|(1+s+|y|)\Gb S(s,y)\|_{L^2_y}\\
&+C\sum_{n=0}^\infty\sum_{|\beta|\le|\alpha|+5}
\|(1+|y|)\phi_n(y)\Gb V(0,y)\|_{L^2_y}.
\end{split}
\]
By definition $V=U+W$, we immediately have estimates for
$(1+t+|x|)\Ga U$. After taking summation over all $\alpha$'s with
$|\alpha|\le k-25$, we arrive at \be\label{estimate:infty}
\begin{split}\LGN U\RGN{k-25,-1}(t)\le& \LGN W\RGN{k-25,-1}+
C\left(\sum_{n=0}^\infty\sum_{|\beta|\le k-21}
\sup_{s\in(0,t)}\phi_n(s)\|(1+s+|y|)\Gb S(s,y)\|_{L^2_y}\right.\\
&\left.+\|U(0,y)\|_{\Hw^{k+1,k}}+\sum_{|\beta|\le k-20}\|(1+|y|)\Gb
W(0,y)\|_{L^2_y}\right).\end{split} \ee

To obtain estimate on each term, we use the following proposition,
the proof of which is given in the Appendix.
\begin{proposition}\label{prop:est}
Let the scalar function $q:\mR^2\times\mR^2\mapsto\mR$ satisfy the
growth condition (\ref{growth:Q}) in terms of its Fourier transform.
Let $f(t,x)$, $g(t,x)$, $h(t,x)$ be functions with sufficient
regularity. Consider $p=\infty$ (respectively $p=2$). Then, at each
$t\ge0$, for $a=|\beta|+8$ (respectively $a=|\beta|+6$) and
$b=\left\lceil {|\beta|\over2}\right\rceil+7$,
\begin{align}\label{Estimate1prop}
\left\|(1+t+|x|)\Gb[f,q,g]\right\|_{L^p_x}\le& C\left(\LGNN
f\RGNN{a}\LGN g\RGN{b,-1}+\LGN f\RGN{b,-1}\LGNN g\RGNN{a}\right),
\end{align}
\begin{align}\label{Estimate2prop}
\begin{split}
\left\|(1+t+|x|)\Gb[f,q,gh]\right\|_{L^p_x}\le &C(1+t)^{-1}
\left(\LGNN f\RGNN{a}\LGN g\RGN{b,-1}\LGN h\RGN{b,-1}\right.\\
&+\left.\LGN f\RGN{b,-1}\left(\LGNN g\RGNN{a}\LGN
h\RGN{\lceil{a\over2}\rceil,-1} +\LGN
g\RGN{\lceil{a\over2}\rceil,-1}\LGNN h\RGNN{a}\right)\right).
\end{split}
\end{align}
\end{proposition}

Apply this proposition (with $p=\infty$) on the first term, and
(with $p=2$) on the second and fourth terms of the RHS of
(\ref{estimate:infty}) and use the definition of $W$ in
(\ref{def:vi}) and $S$ in (\ref{cubicKG}),
\[\begin{split}\LGN U\RGN{k-25,-1}(t)\le& CX^2(t)+C\sum_{n=0}^\infty
\sup_{s\in(0,t)}\phi_n(s)(1+s)^{-1}(X^3(t)+X^4(t))\\
&+ C\left(\|U_0\|_{\Hw^{k+1,k}}+\|\vU_0\|^2_{\Hw^{k+1,k}}\right).
\end{split}\]

We finish the $L^\infty$ estimate part of this Lemma
using the fact that \[\sum_{n=0}^\infty\sup_{s\in
(0,t)}\phi_n(s)(1+s)^{-1}<{5\over2}\]and the assumptions
$\|\vU_0\|_{\Hw^{k+1,k}}\le1$, $X(t)\le1$.

\emph{Step 3.} We obtain the $L^2$ estimate part regarding the terms
\(\LGNN\vU\RGNN{k-9}(t)+\LGNN\tpa \vU\RGNN{k-9}(t)\) using a very
similar approach as in Step 2. In fact, apply $\Ga$ on the
Klein-Gordon equation (\ref{cubicKG}) and use the commutation
properties of vector fields to obtain \((\ptt -\Delta +1)\Ga V=\Ga
S\). Then, we take the inner product of this equation with $\pt\Ga
V$, sum over all $\alpha$ with $|\alpha|\le k-9$ to obtain
\[
\begin{split}\LGNN V(t,x)\RGNN{k-9}+&\LGNN \pa V(t,x)\RGNN{k-9}\le
C\int_0^t\LGNN S(s,x)\RGNN{k-9}ds\\&+C\left(\LGNN V(0,x)\RGNN{k-9}
+\LGNN \pa V(0,x)\RGNN{k-9}\right).\end{split}
\]
Since $V=U+W$, we have
\[
\begin{split}\LGNN U(t,x)\RGNN{k-9}+&\LGNN \pa U(t,x)\RGNN{k-9}
\le\left(\LGNN W(t,x)\RGNN{k-9}+\LGNN \pa W(t,x)\RGNN{k-9}\right)\\&
+C\int_0^t\LGNN S(s,x)\RGNN{k-9}ds\\&+ C\left(\LGNN
U(0,x)\RGNN{k-9}+\LGNN \pa U(0,x)\RGNN{k-9}+
\LGNN W(0,x)\RGNN{k-9}+\LGNN \pa W(0,x)\RGNN{k-9}\right)\\
&=:I+II+III.\end{split}
\]

The estimate of the $I,II,III$ terms above follows closely to that of
(\ref{estimate:infty}), which evokes Proposition \ref{prop:est} repeatedly,
\[\begin{split}I\le&(1+t)^{-1}\sum_{|\beta|\le k-9}\left\|(1+t+|x|)\Gb W(t,x)
\right\|_{L^2_x}\\
\le&C(1+t)^{-1}\LGNN \vU(t,x)\RGNN{k}\LGN \vU(t,x)\RGN{k-25,-1}
\quad\mbox{ by (\ref{def:vi}) and Prop. \ref{prop:est}}\\
\le&C(1+t)^{-1+\sigma}X^2(t),\\
II\le&\int_0^t(1+s)^{-1}\sum_{|\beta|\le k-9}\left\|(1+s+|x|)\Gb
S(s,x)
\right\|_{L^2_x}ds\\
\le&C\int_0^t(1+s)^{-2}\LGNN \vU(s,x)\RGNN{k} \left(\LGN
\vU(s,x)\RGN{k-25,-1}^2
+\LGN \vU(s,x)\RGN{k-25,-1}^3\right)\\
&\phantom{C\int_0^t(1+s)^{-2+\sigma}\left(X^3(s)+X^4(s)\right)\,ds}
\quad\mbox{ by (\ref{cubicRHS}) and Prop. \ref{prop:est}}\\
\le&C\int_0^t(1+s)^{-2+\sigma}\left(X^3(s)+X^4(s)\right)\,ds,\\
III\le&C\left(\|U(0,y)\|_{\Hw^{k+1,k}}+\|\vU(0,y)\|^2_{\Hw^{k+1,k}}\right),
\end{split}
\]
where we use $k\geq 52$ and that $S$ contains at most third order
derivatives of $\vU$. These estimates shall finish the proof of
Lemma \ref{lemL} given assumptions 
$\|\vU_0\|_{\Hw^{k+1,k}}\le1$ and $X(t)\le1$.
\end{proof}

Note that in the above estimates for $I$ and $II$, we use the $k-th$
order norms to bound all lower order norms. In order to get the
global a priori estimate, we have to close the estimates for the
highest order norms. For the RSW system, this is achieved by the
energy estimates on the highest order Sobolev norms
$\LGNN\cdot\RGNN{k}$,
 where its symmetric structure shown in Lemma
\ref{symm} plays a crucial role.
\begin{lemma}\label{lemenergy}
Assume $\|A_{ij}(\vU)\|_{L^\infty}\leq 1/4$.
Then, the solution $\vU$ of (\ref{ptt}) satisfies
\begin{align*}\label{Enestimate}
(1+t)^{-\sigma}(\LGNN\vU\RGNN{k}(t)+\LGNN\tpa \vU\RGNN{k}(t))\leq
C(\|\vU_0\|_{H^{k+1,k}}+X^2(t))
\end{align*}
as long as the solution exists. Here, constant $C$ is independent of $\delta$ and $t$.
\end{lemma}

\begin{proof}
The proof of this lemma combines the ideas in \cite{Hormander,
OzawaSL, Fang} for energy estimates for the Klein-Gordon equations
together.

Define an energy functional %\[ {\hgaE}(t):=\frac{1}{2}\sum_{|\alpha|\leq k} (\|\pa_t\Ga
%\vU\|_{L^2}^2+\|\nabla \Ga \vU\|_{L^2}^2+\|\Ga\vU\|_{L^2}^2)(t) \]and 
\begin{align*} \gaE(t):=&\frac{1}{2}\sum_{|\alpha|\leq k} (\|\pa_t\Ga
\vU\|_{L^2}^2+\|\nabla \Ga \vU\|_{L^2}^2+\|\Ga\vU\|_{L^2}^2)(t)\\
&+\frac{1}{2}\sum_{|\alpha|\leq k}
\sum_{i,j=1}^2\lang A_{ij}(\vU)\pa_i\Ga \vU, \pa_j\Ga\vu\rang(t),
\end{align*} where $\langle, \rangle$ is the $L^2$ inner product defined by
$\lang f, g\rang=\int_{\mR^2}f^Tgdx$.
Clearly, by the commutation property of $\Gamma$, $\pa$ and the assumption $\|A_{ij}(\vU)\|_{L^\infty}\leq
\frac{1}{4}$, we have \be\label{MEE} C_1\sqrt{\gaE(t)}\leq \LGNN\vU\RGNN{k}(t)+\LGNN\tpa \vU\RGNN{k}(t)\leq C_2\sqrt{\gaE(t)}, \ee which ensures the equivalence
of $\sqrt{\gaE(t)}$ and $\LGNN\vU\RGNN{k}(t)+\LGNN\tpa \vU\RGNN{k}(t)$.

Applying $\Gamma^{\alpha}$ to (\ref{ptt})  and taking the $L^2$
inner product on  the resulting system with $\pt\Gamma^{\alpha}\vU$
i.e. $\po\Gamma^{\alpha}\vU$ , it follows from Leibniz's rule that
\begin{align*}
\lang(\ptt-\Delta+1)\Gamma^{\alpha}\vU, \pt\Gamma^{\alpha}\vU\rang=
&\sum_{i,j=1}^2\lang
A_{ij}(\vU)\pa_{ij}\Gamma^{\alpha}\vU,
\po\Gamma^{\alpha}\vU\rang \\+\sum_{j=1}^2\lang A_{0j}
\pa_{0j}\Gamma^{\alpha}\vU,
\po\Gamma^{\alpha}\vU\rang&+\sum_{|\beta|+|\gamma|\le |\alpha|}\lang
R^\bega(\Gb\vUd\otimes\Gg\vUd),
\po\Ga\vU \rang,
\end{align*}
where all $R^\bega$'s are linear functions. Here
$\vUd=(\vU^T,\pt\vU^T,\px\vU^T,\py\vU^T)$.

Upon integrating by parts, one has
\begin{equation}\label{Highenergyest}
\begin{split}
&\frac{1}{2}\pa_t(\|\pa_t\Ga \vU\|_{L^2}^2+\|\nabla \Ga \vU\|_{L^2}^2
+\|\Ga\vU\|_{L^2}^2)\\
=&-\sum_{i,j=1}^2 \lang
A_{ij}\pa_i\Ga\vU,\pa_0\pa_j\Ga\vU\rang -\lang
\pa_jA_{ij}(\vU)\pa_i\Ga\vU,\pa_0\Ga\vU\rang\\
& -\frac{1}{2}\sum_{j=1}^2\lang\pa_j
A_{0j}\pa_0\Ga\vU,\pa_0\Ga\vU\rang +\sum_{|\beta|+|\gamma|\le
|\alpha|}\lang R^\bega(\Gb\vUd\otimes\Gg\vUd),
\po\Ga\vU\rang
\end{split}
\end{equation}
and the first terms on the RHS above\begin{align*} \lang
A_{ij}\pa_i\Ga\vU,\pa_0\pa_j\Ga\vU\rang=&-\frac{1}{2}\pa_t\lang
A_{ij}\pa_i\Ga\vU,\pa_j\Ga\vU\rang+\frac{1}{2}\lang \pa_0A_{ij}
\pa_i\Ga\vU,\pa_j\Ga\vU\rang.
\end{align*}
The equations above hold because $A_{ij}$ and $A_{0j}$ are symmetric
matrices.  Summing for all indices $\alpha$ with $|\alpha|\leq k$ in
the above equations and using the fact that $\min\{|\beta|,|\gamma|\}\le k/2<k-25$ in the last terms on the RHS of
(\ref{Highenergyest}), one obtain
\[\pa_t \gaE(t)\leq C\LGN\vU\RGN{k-25,0}\gaE(t).\]
Thus,
\[
\pa_t \sqrt{\gaE(t)}\leq C(1+t)^{-1}\LGN\vU\RGN{k-25,-1}(1+t)^\sigma
(1+t)^{-\sigma}\sqrt{\gaE(t)}.
\]

By the virtue of (\ref{MEE}) and the definition of $X(t)$ in (\ref{def:X}), we have
\begin{align*}
\LGNN\vU\RGNN{k}+\LGNN\tpa \vU\RGNN{k}\leq &C\left(\sqrt{\gaE(t)}-\sqrt{\gaE(0)}\right)+C\sqrt{\gaE(0)}\\
\leq&
C\int_0^t(1+s)^{-1}{X(s)}(1+s)^{\sigma}
X(s)ds +C\|\vU_0\|_{H^{k+1,k}}\\
\leq& C(1+t)^{\sigma}X^2(t)+C\|\vU_0\|_{H^{k+1,k}},
\end{align*}
which finishes the proof of Lemma \ref{lemenergy}.
\end{proof}

The proofs of these lemmas also help reveal the asymptotic behavior of $U$ in the following theorem
\begin{theorem}\label{asymTh}
Under the same assumptions as in Theorem \ref{KGsystem}, there exists a free solution $\vU^+$
such that \be\label{KGasym} 
\|\vU(t,\cdot)-\vU^+(t,\cdot)\|_{H^{k-15}}+\|\pa_t\vU(t,\cdot)-\pa_t\vU^+(t,\cdot)\|_{H^{k-16}}\leq C(1+t)^{-1}, \ee where
\be\label{free:U}\vU^+(t,\cdot):=\cos((1-\Delta)^{1/2}
t)\vU_0^++(1-\Delta)^{-1/2}\sin((1-\Delta)^{1/2}
t)\vU_1^+\ee for some initial data $\vU_{0}^+\in H^{k-15}$ and $\vU_1^+\in H^{k-16}$.
\end{theorem}
\begin{proof}
Recall the definitions of $\vV$ in (\ref{def:V}), $\vW$ in (\ref{def:vi}) and $S$ in (\ref{cubicKG}), (\ref{cubicRHS}). Without loss of generality, switch the notations in (\ref{cubicKG}) to boldface $\vV$ and $\vS$ while cubic nonlinearity of $\vS$ remains valid. Then, Theorem \ref{KGsystem} and Proposition \ref{prop:est} imply that
\be\label{est:S}
   \|\vS(s,\cdot)\|_{H^{k-15}}\leq C(1+s)^{-2}.
   \ee
   Therefore, we can define
\be\label{free:U0}
 \vU^+_{0}:=\vV(0,\cdot)-\int_0^{\infty}(1-\Delta)^{-1/2}\sin
 ((1-\Delta)^{1/2}s)\vS(s,\cdot) ds\in H^{k-15}
 \ee
and
\be\label{free:U1}
 \vU^+_{1}:=\pa_t\vV(0,\cdot)+\int_0^{\infty}\cos
 ((1-\Delta)^{1/2}s)\vS(s,\cdot) ds\in H^{k-16}.
 \ee
 
Then, apply the Duhamel's principle on (\ref{cubicKG}),
\begin{align*} \vV(t,\cdot) =&\cos((1-\Delta)^{1/2}
t)\vV(0,\cdot)+(1-\Delta)^{-1/2}\sin((1-\Delta)^{1/2}
t)\pa_t\vV(0,\cdot)\\
&+\int_0^t(1-\Delta)^{-1/2}\sin((1-\Delta)^{1/2}(t-s))\vS(s,\cdot)ds,
\end{align*}
and combine it with (\ref{free:U}), (\ref{free:U0}), (\ref{free:U1}) so that
we have
\begin{align*}
\vV(t,\cdot)-\vU^+(t,\cdot)=&\int_t^{\infty}(1-\Delta)^{-1/2}\sin
 ((1-\Delta)^{1/2}(t-s))\vS(t,\cdot)ds.
\end{align*}
Therefore, by (\ref{est:S}), \[\|\vV(t,\cdot)-\vU^+(t,\cdot)\|_{H^{k-15}}+\|\pt\vV(t,\cdot)-\pt\vU^+(t,\cdot)\|_{H^{k-16}}\le C(1+t)^{-1}\]

Finally, apply Theorem \ref{KGsystem} and Proposition \ref{prop:est} to $\vW$ defined in (\ref{def:vi}) and arrive at
\[
   \|\vW(t,\cdot)\|_{H^{k-15}}+\|\pt\vW(t,\cdot)\|_{H^{k-16}}\leq C(1+t)^{-1}.
   \]
We then conclude that $\vU-\vU^+=\vV-\vU^+-\vW$ also decays like $C(1+t)^{-1}$ as in (\ref{KGasym}).
\end{proof}

\section{Lifespan of Classical Solutions with General Initial Data}\label{sec:general}
The proof of Theorem \ref{perTh} combines standard energy methods with the sharp estimates from previous section, in particular the $(1+t)^{-1}$ decay rate of $L^\infty$ norms.

We start with extracting the zero-relative-vorticity part $(\rhoK_0,\vuK_0)$ from the general initial data $(\rho_0,\vu_0)$ --- supscript ``$K$'' here stands for Klein-Gordon. This can be achieved by $L^2$ projection of $(\rho_0,\vu_0)$ onto the function space with zero relative vorticity but we find it easier to use the complementary projection,
\begin{align*}\rho_0-\rhoK_0:=&(\Delta-1)^{-1}(\pa_1
u_{2,0}-\pa_2u_{1,0}-\rho_0),\\
\vu_0-\vuK_0:=&\left(\begin{matrix}-\pa_2\\ \pa_1\end{matrix}\right)(\rho_0-\rhoK_0).\end{align*}

Using the notations from Theorem \ref{perTh}, we have
\begin{align}\label{size:E0}\|(\rho_0,\vu_0)-(\rhoK_0,\vuK_0)\|_{H^3}\le& C\|\pa_1
u_{2,0}-\pa_2u_{1,0}-\rho_0\|_{H^2}=C\ep,\\
\label{size:K0}\|(\rhoK_0,\vuK_0)\|_{H^{k+1,k}}\le& C\|(\rho_0,\vu_0)\|_{H^{k+1,k}}=C\delta.
\end{align}

Let $m^{K}:=2(\sqrt{1+\rho^K}-1)$, then $\vU^{K}:=(m^K, u_1^K, u_2^K)^T$
solves the symmetrized RSW system (\ref{RSWm}), (\ref{RSWum}) as well as the Klein-Gordon system (\ref{ptt}). By choosing $\delta$ in (\ref{size:K0}) to be sufficiently small, we have $\vU^K_0$ satisfy the assumptions in
Theorem \ref{KGsystem}, so that there exists a unique global solution
$\vU^K$ to (\ref{RSWm}) and (\ref{RSWm})
associated with the initial data $\vU_0^K$. In addition, the following estimate holds true\be\label{Irest}
 |\vU^{K}|_{W^{k-25,\infty}}\leq \frac{C\delta}{1+t}.
\ee

Now it remains to estimate the difference $\vE:=\vU-\vU^K$. To this end, we write the symmetrized RSW system (\ref{RSWm}), (\ref{RSWum}) into the following compact form \be\label{cptfeq}
\pa_t\vU+\sum_{j=1}^2B_j(\vU)\pa_j\vU=\cL(\vU), \ee with symmetric matrices
\[ B_j(\vU):=u_jI+\frac{1}{2}m J_j\]
and $J_j$, $\cL$ defined in (\ref{defJL}).

Since both $\vU=\vU^K+\vE$ and $\vU^{K}$ satisfy the same system (\ref{cptfeq}), straightforward calculation shows that $\vE$ satisfies
\[\pa_t\vE+\sum_{j=1}^2B_j(\vE)\pa_j\vE+\sum_{j=1}^2B_j(\vU^K)\pa_j\vE+\sum_{j=1}^2B_j(\vE)\pa_j\vU^K=\cL(\vE),\]
subject to initial data $\vE_0=(2\sqrt{1+\rho_0}-2\sqrt{1+\rhoK_0}, \vu_0-\vuK_0)$.

Employing the standard energy method, we have $e(t):=\|\vE(t,\cdot)\|_{H^3}$ satisfy an energy inequality,
\[e'(t)\le Ce(t)(|\nabla\vE|_{L^\infty}+|\vU^K|_{W^{4,\infty}}).\]
We then use Sobolev inequalities and estimate (\ref{Irest}) to further the above inequality as
\[e'(t)\le Ce(t)\left(e(t)+{\delta\over1+t}\right).\]
Divide it with $e^2(t)$,
\[\left(-e^{-1}(t)\right)'\le C\left(1+{\delta\over1+t}e^{-1}(t)\right)\]which is linear in terms of $e^{-1}(t)$. We finally arrive at
\[e(t)\le(1+t)^{C\delta}\left(e^{-1}(0)-{C\over 1+C\delta}\left[(1+t)^{1+C\delta}-1\right]\right)^{-1}.\]

By (\ref{size:E0}), the initial value is bounded by $e(0)\le C\ep$. Thus, $e(t)$ remains bounded as long as\[t<\left({1+C\delta\over C\veps}+1\right)^{1\over1+C\delta}-1\]
which is of the same order as (\ref{span:general}) in Theorem \ref{perTh} under the smallness assumption on $\delta$ and the fact that $\veps\le\delta$.

\section{Appendix: Proof of Proposition \ref{prop:est}}
We first claim the following estimate on kernel $q(y,z)$.
\begin{proposition}\label{prop:kernel}Let $q(y,z)$ satisfy the growth
condition (\ref{growth:Q}), i.e.
\be\label{growth:q}|D^N\hat{q}(\xi,\eta)|\le
C_N(1+|\xi|^4+|\eta|^4)\ee for any $N\geq 0$. Then, for
\be\label{def:q1}q_1:=(1-\Delta_y)^{-3}(1-\Delta_z)^{-3}q(y,z)\ee
the following estimate holds true
\[
(1+|y|+|z|)^lq_1(y,z)\in L^1(\mR^2_y\times\mR^2_z),
\]
for any $l\geq 0$.
 \end{proposition}
\begin{proof}By (\ref{growth:q}), we have a growth condition for $q_1$,
\[|\hat{q}_1(\xi,\eta)|={|\hat{q}(\xi,\eta)|\over(1+|\xi|^2)^3(1+|\eta|^2)^3}
\le C(1+|\xi|^2)^{-1}(1+|\eta|^2)^{-1}\] and inductively,
\[|D^{N}\hat{q}_1(\xi,\eta)|\le C(1+|\xi|^2)^{-1}(1+|\eta|^2)^{-1}\in
L^2(\mR^2_\xi\times\mR^2_\eta)\] for any integer $N\ge0$. Therefore, by the Plancherel Theorem
$(1+|y|^N+|z|^N)q_1(y,z)\in L^2_{yz}$, which readily implies
\[
(1+|y|+|z|)^l q_1(y,z)=(1+|y|+|z|)^{l-N}\cdot
(1+|y|+|z|)^Nq_1(y,z)\in L^1(\mR^2_y\times\mR^2_z).
\]
\end{proof}

Note that (\ref{Estimate2prop}) is a direct consequence of
(\ref{Estimate1prop}), it suffices to prove (\ref{Estimate1prop}).
We then show the following estimate that serves as a slightly
stronger version of Proposition \ref{prop:est}: for any kernel
$q(y,z)$ satisfying the growth condition (\ref{growth:q}) and
$f(t,x)$, $g(t,x)$ with sufficient regularity, there exists a
constant $C$ independent of $f,g$ such that
\be\label{prop:proof}\left\|(1+t+|x|)\Gb[f,q,g]\right\|_{L^p_x}\le
C\sum_{i+j=|\beta|}\min\left\{\LGNN f\RGNN{i+\gamma}\LGN
g\RGN{j+6,-1},\,\LGN f\RGN{i+6,-1}\LGNN g\RGNN{j+\gamma}\right\}\ee
where $\gamma=6$ if $p=2$ or $\gamma=8$ if $p=\infty$.

We prove it by induction.

\emph{Step 1.} Set $|\beta|=0$. By integrating by parts, the LHS of
(\ref{prop:proof})
\be\label{step0:1}\left\|(1+t+|x|)[f,q,g]\right\|_{L^p_x}
=\left\|(1+t+|x|)[f_1,q_1,g_1]\right\|_{L^p_x}\ee where $q_1$ is defined in (\ref{def:q1}) and
\[\begin{split} f_1(t,y):=&(1-\Delta_y)^3f(t,y),\qquad
%q_1(y,z):=&(1-\Delta_y)^{-3}(1-\Delta_z)^{-3}q(y,z)\\
g_1(t,z):=(1-\Delta_z)^3g(t,z)\end{split}.\]
Continue from (\ref{step0:1}),
\[\begin{split}
&\left\|(1+t+|x|)[f,q,g]\right\|_{L^p_x}\\
=
&\left\|(1+t+|x|)\int_{\mR^2\times\mR^2}f_1(t,x-y)
q_1(y,z)g_1(t,x-z)\,dydz\right\|_{L^p_x}\\
\le&\left\|\int_{\mR^2\times\mR^2}\left|(1+t+|x-y|)
f_1(t,x-y)q_1(y,z)g_1(t,x-z)\right|\,dydz\right\|_{L^p_x}\\
&+\left\|\int_{\mR^2\times\mR^2}\left|f_1(t,x-y)yq_1(y,z)
g_1(t,x-z)\right|\,dydz\right\|_{L^p_x}\\
\le&\left|(1+t+|y|)f_1(t,y)\right|_{L^\infty_y}\|q_1(y,z)\|_{L^1_{yz}}
\|g_1(t,z)\|_{L^p_z}\\&+\left|f_1(t,y)\right|_{L^\infty_y}
\|yq_1(y,z)\|_{L^1_{yz}}\|g_1(t,z)\|_{L^p_z}
\end{split}\]
by Young's inequality. Combined with Proposition \ref{prop:kernel}
(and Sobolev inequality if $p=\infty$), this
implies\[\left\|(1+t+|x|)[f,q,g]\right\|_{L^p_x}\le C\LGN
f_1\RGN{0,-1}\|g_1\|_{L^p}\le C\LGN f\RGN{6,-1}\LGNN
g\RGNN{\gamma}.\]The same estimate holds if we switch $f$ and $g$.
Thus, we proved (\ref{prop:proof}) for $|\beta|=0$.

\emph{Step $2$.} Suppose (\ref{prop:proof}) is true for all $(n-1)$-th order
vector fields. Now pick any $n$-th order vector field
$\Gb:=\Gamma^{\beta'}\Gamma^1$ where $|\beta'|=n-1$ and
$\Gamma^1\in\{\pa_t,\,\pa_1,\,\pa_2,\,t\pa_1+x_1\pa_t,\,t\pa_2+x_2\pa_t,
\,x_1\pa_2-x_2\pa_1\}$. By product rule and the definition of normal
forms (\ref{def:vi}), for any $\pa\in\{\pa_t,\pa_1,\pa_2\}$,
we have
\[
\begin{split}\pa[f,q,g]=&[\pa f,q,g]+[f,q,\pa g],\\
t\pa[f,q,g]=&[t\pa f,q,g]+[f,q,t\pa g],\\
x_i\pa[f,q,g]=&\left([x_i\pa f(t,x),q(y,z),g(t,x)]
+[\pa f(t,x),y_iq(y,z),g(t,x)]\right)\\
&+\left([f(t,x),q(y,z),x_i\pa g(t,x)]+[f(t,x),z_iq(y,z),\pa g(t,x)]\right),
\end{split}
\]
which immediately implies that
\be\label{productrule}
\begin{split}\Gb[f,q,g]=&\Gamma^{\beta'}\left([\Gamma^1 f,q,g]
+[f,q,\Gamma^1g]\right)\\ &+\Gamma^{\beta'}\sum_{i=0}^2
\sum_{j=1}^2C_{ij}\left([\pa_if,y_jq,g]+[f,z_jq,\pa_ig]\right).
\end{split}
\ee Here, the kernels are $q(y,z)$, $y_jq(y,z)$, $z_jq(y,z)$, all
satisfying the growth condition (\ref{growth:q}). Therefore, the
inductive hypothesis is true and we apply (\ref{prop:proof}) with
$|\beta'|=n-1$ on (\ref{productrule}) to conclude that
(\ref{prop:proof}) also holds for $|\beta|=n$. This finishes the proof.

\providecommand{\bysame}{\leavevmode\hbox to3em{\hrulefill}\thinspace}
\providecommand{\MR}{\relax\ifhmode\unskip\space\fi MR }
% \MRhref is called by the amsart/book/proc definition of \MR.
\providecommand{\MRhref}[2]{%
  \href{http://www.ams.org/mathscinet-getitem?mr=#1}{#2}
}
\providecommand{\href}[2]{#2}


\begin{thebibliography}{10}

\bibitem{Babin}
A.~Babin, A.~Mahalov, and B.~Nicolaenko, \emph{Global splitting and regularity
  of rotating shallow-water equations}, European J. Mech. B Fluids \textbf{16}
  (1997), no.~5, 725--754.

\bibitem{Bresch}
Didier Bresch, Beno{\^{\i}}t Desjardins, and Guy M{\'e}tivier, \emph{Recent
  mathematical results and open problems about shallow water equations},
  Analysis and simulation of fluid dynamics, Adv. Math. Fluid Mech.,
  Birkh\"auser, Basel, 2007, pp.~15--31.

\bibitem{ChTa:SIAM}
Bin Cheng and Eitan Tadmor, \emph{Long-time existence of smooth solutions for
  the rapidly rotating shallow-water and {E}uler equations}, SIAM J. Math.
  Anal. \textbf{39} (2008), no.~5, 1668--1685.

\bibitem{Fang}
Jean-Marc Delort, Daoyuan Fang, and Ruying Xue, \emph{Global existence of small
  solutions for quadratic quasilinear {K}lein-{G}ordon systems in two space
  dimensions}, J. Funct. Anal. \textbf{211} (2004), no.~2, 288--323.

\bibitem{GeorgievP}
V.~Georgiev and P.~Popivanov, \emph{Global solution to the two-dimensional
  {K}lein-{G}ordon equation}, Comm. Partial Differential Equations \textbf{16}
  (1991), no.~6-7, 941--995.

\bibitem{Georgiev}
Vladimir Georgiev, \emph{Decay estimates for the {K}lein-{G}ordon equation},
  Comm. Partial Differential Equations \textbf{17} (1992), no.~7-8, 1111--1139.

\bibitem{Guo}
Yan Guo, \emph{Smooth irrotational flows in the large to the {E}uler-{P}oisson
  system in {$\bold R\sp {3+1}$}}, Comm. Math. Phys. \textbf{195} (1998),
  no.~2, 249--265.

\bibitem{Hormander}
Lars H{\"o}rmander, \emph{Lectures on nonlinear hyperbolic differential
  equations}, Math\'ematiques \& Applications (Berlin) [Mathematics \&
  Applications], vol.~26, Springer-Verlag, Berlin, 1997.

\bibitem{Kato}
Tosio Kato, \emph{The {C}auchy problem for quasi-linear symmetric hyperbolic
  systems}, Arch. Rational Mech. Anal. \textbf{58} (1975), no.~3, 181--205.

\bibitem{KlainermanP}
S.~Klainerman and Gustavo Ponce, \emph{Global, small amplitude solutions to
  nonlinear evolution equations}, Comm. Pure Appl. Math. \textbf{36} (1983),
  no.~1, 133--141.

\bibitem{Klainerman}
Sergiu Klainerman, \emph{Global existence of small amplitude solutions to
  nonlinear {K}lein-{G}ordon equations in four space-time dimensions}, Comm.
  Pure Appl. Math. \textbf{38} (1985), no.~5, 631--641.

\bibitem{LiTa:rotation}
Hailiang Liu and Eitan Tadmor, \emph{Rotation prevents finite-time breakdown}, Phys. D \textbf{188}  (2004), no.~3-4, 262--276.

\bibitem{Majda84}
Andrew Majda, \emph{Compressible fluid flow and systems of conservation laws in
  several space variables}, Applied Mathematical Sciences, vol.~53,
  Springer-Verlag, New York, 1984.

\bibitem{Majda}
\bysame, \emph{Introduction to {PDE}s and waves for the atmosphere and ocean},
  Courant Lecture Notes in Mathematics, vol.~9, New York University Courant
  Institute of Mathematical Sciences, New York, 2003.

\bibitem{OzawaSL}
Tohru Ozawa, Kimitoshi Tsutaya, and Yoshio Tsutsumi, \emph{Global existence and
  asymptotic behavior of solutions for the {K}lein-{G}ordon equations with
  quadratic nonlinearity in two space dimensions}, Math. Z. \textbf{222}
  (1996), no.~3, 341--362.

\bibitem{OzawaQL}
\bysame, \emph{Remarks on the {K}lein-{G}ordon equation with quadratic
  nonlinearity in two space dimensions}, Nonlinear waves ({S}apporo, 1995),
  GAKUTO Internat. Ser. Math. Sci. Appl., vol.~10, Gakk\=otosho, Tokyo, 1997,
  pp.~383--392. \MR{MR1602662 (2000e:35155)}

\bibitem{Pedlosky}
J.~Pedlosky, \emph{Geophysical fluid dynamics}, Springer Verlag, Berlin, 1992.

\bibitem{Rammaha}
M.~A. Rammaha, \emph{Formation of singularities in compressible fluids in
  two-space dimensions}, Proc. Amer. Math. Soc. \textbf{107} (1989), no.~3,
  705--714.

\bibitem{Shatahev}
Jalal Shatah, \emph{Global existence of small solutions to nonlinear evolution
  equations}, J. Differential Equations \textbf{46} (1982), no.~3, 409--425.

\bibitem{Shatah}
\bysame, \emph{Normal forms and quadratic nonlinear {K}lein-{G}ordon
  equations}, Comm. Pure Appl. Math. \textbf{38} (1985), no.~5, 685--696.

\bibitem{Sideris:3D:singularity}
Thomas~C. Sideris, \emph{Formation of singularities in three-dimensional
  compressible fluids}, Comm. Math. Phys. \textbf{101} (1985), no.~4, 475--485.

\bibitem{Sideris:3D}
\bysame, \emph{The lifespan of smooth solutions to the three-dimensional
  compressible {E}uler equations and the incompressible limit}, Indiana Univ.
  Math. J. \textbf{40} (1991), no.~2, 535--550.

\bibitem{Sideris:2D}
\bysame, \emph{Delayed singularity formation in {$2$}{D} compressible flow},
  Amer. J. Math. \textbf{119} (1997), no.~2, 371--422.

\bibitem{Simon}
Jacques C.~H. Simon and Erik Taflin, \emph{The {C}auchy problem for nonlinear
  {K}lein-{G}ordon equations}, Comm. Math. Phys. \textbf{152} (1993), no.~3,
  433--478.

\end{thebibliography}
\end{document}